\documentclass[12pt]{article}
\usepackage{graphicx}
\usepackage{epstopdf}
\usepackage{a4wide,color,hyperref} 
\usepackage{amsmath,amssymb,amsthm}

\newcommand{\Pa}{\mathcal{P}}
\newcommand{\La}{\mathcal{L}}
\newcommand{\Qa}{\mathcal{Q}}

\newtheorem{theorem}{Theorem}
\newtheorem{lemma}[theorem]{Lemma}
\newtheorem{corollary}[theorem]{Corollary}
\newtheorem{conjecture}[theorem]{Conjecture}

\begin{document}
\title{Bounding the distance among longest paths in a connected graph}
\author{Jan Ekstein\thanks{Department of Mathematics, Institute for Theoretical Computer Science, and 
        European Centre of Excellence NTIS - New Technologies for the Information Society, Faculty of Applied Sciences, University of West Bohemia, Pilsen, 
        Technick\'a 8, 306 14 Plze\v n, Czech Republic, e-mail: \texttt{$\{$ekstein,kabela,teska$\}$@kma.zcu.cz}.}\and
        Shinya Fujita\thanks{International College of Arts and Sciences, Yokohama City University, 22-2 Seto, Kanazawa-ku, Yokohama 236-0027, Japan,
        e-mail: \texttt{shinya.fujita.ph.d@gmail.com}.}\and
        Adam Kabela\footnotemark[1]\and
        Jakub Teska\footnotemark[1]}
\maketitle

\begin{abstract}
 It is easy to see that in a connected graph any $2$ longest paths have a vertex in common. For $k \geq 7$, Skupie\'n in 1966 obtained a connected 
 graph in which some $k$ longest paths have no common vertex, but every $k-1$ longest paths have a common vertex. It is not known whether every $3$ longest 
 paths in a connected graph have a common vertex and similarly for $4, 5$, and $6$ longest path. Fujita et al. in 2015 give an upper bound on 
 distance among $3$ longest paths in a connected graph. In this paper we give a similar upper bound on distance between $4$ longest paths and also 
 for $k$ longest paths, in general.
\end{abstract}

\section{Introduction}
\label{intro}
In 1966 Gallai in~\cite{gala} asked whether all longest paths in a connected graph have a vertex in common. Couple of years later, several counterexamples
were found, see~\cite{walt},~\cite{waltvoss}, and~\cite{zamf}. In 1976 Thomassen in~\cite{thom} showed that there exist infinitely many counterexamples to 
Gallai's question.

On the other hand, if we restrict to a special class of graphs, the answer to Gallai's question may become positive. For example in a tree, all longest paths 
must have a vertex in common. Klav\v zar and Petkov\v sek in~\cite{klav} proved that it is also true for split graphs and cacti and Balister et al. 
in~\cite{bali} proved it for the class of circular arc graphs.

Another approach to Gallai's question is to ask, what happens if we consider a fixed number of longest paths. It is easy to see that every $2$ longest 
paths in a connected graph have a common vertex. For $3$ longest paths, the question remains open. This has been originally asked by 
Zamfirescu in~\cite{zamf2}.

\begin{conjecture}\emph{\cite{zamf2}}
 \label{3path}
 For every connected graph, any $3$ of its longest paths have a common vertex.
\end{conjecture}

There are few results dealing with this conjecture. Axenovich in~\cite{axen} proved that it is true for connected outerplanar graphs and de Rezende 
et al. in~\cite{reze} showed that Conjecture~\ref{3path} is true for connected graphs in which all nontrivial blocks are hamiltonian.

For $k \geq 7$, Skupie\'n in~\cite{skup} obtained a connected graph in which some $k$ longest paths have no common vertex, but every $k-1$ longest paths 
have a common vertex. Regarding this, it is still valid to ask wheter not only $3$ but also $4, 5$, and $6$ longest path in a connected graph have 
a common vertex. 

In~\cite{FuFuNaOz} the authors introduced a parameter to measure the distance among the longest paths in a connected graph and proved an upper bound of 
this parameter for $3$ longest paths. To state their result we give some definitions first.

Let $G$ be a connected graph. Let $\ell(G)$ be the length of any longest path in $G$ and $\La(G)=\{P$ $|$ $P$ is a path in $G$ with $|V(P)|=\ell(G)+1\}$ 
be a set of longest paths of~$G$. For $x,y \in V(G)$, let $d_G(x,y)$ be the distance between $x$ and $y$ in $G$. For a vertex $x \in V(G)$ and 
a subset $U \subseteq V(G)$, let $d_G(x,U)=min\{d_G(x,y)|$ $y\in U\}$. For $\Pa \subseteq \La(G)$ we call \textit{path-distance-function} 
$f(G,\Pa)=min\{ \sum_{P \in \Pa}d_G(v,V(P))$ $|$ $v\in V(G)\}$.

For a class of graphs $\mathcal{G}$ and an integer $k$, we introduce \textit{path-distance-ratio} $d_{k}(\mathcal{G})=max \frac{f(G,\Pa)}{|V(G)|}$, where 
the maximum is taken over all the graphs of $\mathcal{G}$ and their sets of longest paths $\Pa \subseteq \La(G)$ with $|\Pa|=k$. 

Let $\mathcal{G}_{c}$ be a class of connected graphs. The question whether for every connected graph any $3$ longest paths have a vertex in common translates
into the question whether $d_3(\mathcal{G}_{c}) = 0$. On the other hand, Skupie\'n in~\cite{skup} constructed a graph on $17$ vertices, in which there are 
$7$ longest paths without a common vertex, this graph implies that $d_7(\mathcal{G}_{c}) \geq \frac{1}{17}$.

Now we can state the result by Fujita et al. from~\cite{FuFuNaOz}.

\begin{theorem}\emph{\cite{FuFuNaOz}}
 Let $\mathcal{G}_{c}$ be a class of connected graphs. Then $d_{3}(\mathcal{G}_{c})\leq \frac{1}{17}$.
\end{theorem}

In this paper we prove similar results for $4$ longest path and also for $k$ longest paths, in general.

\begin{theorem}
 \label{4path}
  Let $\mathcal{G}_{c}$ be a class of connected graphs. Then $d_{4}(\mathcal{G}_{c})\leq \frac{3}{16}$.
\end{theorem}

By picking any vertex of a connected graph $G$, we see that $d_k(\mathcal{G}_{c})$ can be bounded by $k$.
We show that it can be improved as roughly $\frac{k}{6}$.

\begin{theorem}
 \label{kpath}
 Let $\mathcal{G}_{c}$ be a class of connected graphs and let $k\geq 3$ be an integer. Then $d_{k}(\mathcal{G}_{c})\leq \frac{k^{3}-4k^{2}+5k-2}{6k^{2}-8k}$.
 \end{theorem}

\section{Proofs}

In our proofs, we adapt ideas of~\cite{FuFuNaOz}. We start by giving several technical definitions.

Let $G$ be a connected graph.
Let $U$ and $V$ be two sets of vertices of $G$, let $P$ be a path in $G$ and $Q$ be a subpath of $P$.
Let $u$ and $v$ be the end-vertices of $Q$,
we say $Q$ is a \emph{$U-V$ path} on $P$ if $u\in U$ and $v\in V$.
A vertex of a path which is not its end-vertex is an \emph{int-vertex} of the path.
Let $uPv$ denote the $\{u\}-\{v\}$~path 
on~$P$. Futhermore, let \emph{\v{u}}$Pv=uPv-u$, $uP$\emph{\v{v}}$=uPv-v$ and \emph{\v{u}}$P$\emph{\v{u}}$=uPv-\{u,v\}$. For a set 
$\Pa=\{P,P_1,P_2,...,P_{k-1}\}\subseteq \La(G)$ and $i\neq j\in \{1,2,...,k-1\}$, a $V(P_i)-V(P_j)$ path $Q$ on $P$ is \emph{good} if 
$V(Q)\cap V(P_m)\neq\emptyset$ for every $m=1,2,...,k-1$ and neither $P_i$ nor $P_j$ contain an int-vertex of $Q$. Let $t_{\Pa}(P)$ be the number of all good 
paths of $P$ and $t'_{\Pa}(P)$ be the maximum number of all non-intersecting (no edge in common) good paths on $P$. 
By Proposition 3 in \cite{FuFuNaOz}, every 2 longest paths intersect.
Thus, we have that $t_{\Pa}(P)\geq t'_{\Pa}(P)\geq 1$ for every $P\in \Pa$. For a path $P \in \Pa$, let $X^i_{\Pa}(P)$ denote
the set of all vertices of $P$ which are exactly on $i$ paths from $\Pa$. Let $n_i=|\bigcup_{P\in \Pa} X^i_{\Pa}(P)|$.
 
 \begin{lemma}
  \label{nbound}
   Let $G$ be a connected graph of order $n$ and $\Pa \subseteq \La(G)$ with $|\Pa|=k\geq 3$. If $f(G,\Pa)>0$, then 
                       $$n\geq \frac{k\cdot \ell(G)+k+(k-2)n_{1}+(k-3)n_{2}+...+n_{k-2}}{k-1}.$$
 \end{lemma}
 
 \begin{proof}
  Clearly $n \geq n_{1}+n_{2}+...+n_{k-1}+n_{k}$, where $n_{k}=0$, and $n\geq k(\ell(G)+1)-n_{2}-2n_{3}-...-(k-3)n_{k-2}-(k-2)n_{k-1}$. 
  Hence $n\geq k\cdot \ell(G)+k-n_{2}-2n_{3}-...-(k-3)n_{k-2}-(k-2)(n-n_{1}-n_{2}-...-n_{k-2})$ and the result follows.
 \end{proof}
 
\begin{lemma}
 \label{Seg1}
 Let $G$ be a connected graph and $\Pa \subseteq \La(G)$ with $|\Pa|=k$. If there exists a path $P\in \Pa$ with $t'_{\Pa}(P)=1$, then $f(G,\Pa)=0$.
\end{lemma}
 
\begin{proof}
 To the contrary, we suppose there is a path $P = v_1 v_2 ... v_{\ell(G)+1}$ with $t'_{\Pa}(P)=1$ and $f(G,\Pa)>0$. By $f(G,\Pa)>0$, every good path on $P$
 contains an edge. We consider the 'left-most' good path $Q$ on $P$; more formally, we consider the good path $Q  = v_i v_{i+1} ... v_j$ such that there
 is no good path on $P$ containing a vertex $v_k$ with $k < i$. Let $\Pa_j$ denote the set of paths of $\Pa$ which contain $v_j$. By the choice of $Q$, 
 some path of $\Pa_j$ contains no vertex $v_k$ with $k < j$, and thus the length of $v_1 v_2 ... v_j$ is at most $\frac{1}{2}\ell(G)$. Similarly, we consider
 the 'right-most' good path $Q' = v_{i'} v_{i'+1} ... v_{j'}$ and we see that the length of $v_{i'} v_{i'+1} ... v_{\ell(G)+1}$ is at most 
 $\frac{1}{2}\ell(G)$. By the assumption $t'_{\Pa}(P)=1$, the paths $Q$ and $Q'$ have an edge in common, so $j > i'$, hence the length of $P$ is shorter 
 than $\ell(G)$, a contradiction.
\end{proof}
 
\begin{lemma}
 \label{X1X2}
 Let $G$ be a connected graph and $\Pa \subseteq \La(G)$ with $|\Pa|=k\geq3$. Let $P\in \Pa$ and let $Q$ be a good path on $\Pa$.  
 Then the following two statements hold:
 \begin{itemize}
  \item[(i)] $f(G,\Pa)\leq\frac{|V(Q)|-1}{2}(k-1)$;
  \item[(ii)]  $|X^1_{\Pa}(P)\cup X^2_{\Pa}(P)\cup ...\cup X^{k-2}_{\Pa}(P)|\geq t'_{\Pa}(P)(\frac{2}{k-1}f(G,\Pa)-1)$. 
 \end{itemize}  
\end{lemma}
 
\begin{proof}
 Note that if $f(G,\Pa) = 0$, then the statement holds. Suppose $f(G,\Pa)\geq 1$. In particular, every good path on $\Pa$ contains at least two vertices. 
 Let $x\in V(Q)$ such that $\sum_{P'\in \Pa}d_{G}(x,P')\leq\sum_{P'\in \Pa}d_{G}(y,P')$ for every $y\in V(Q)$. Then 
 $$f(G,\Pa)\leq \sum_{P'\in \Pa}d_{G}(x,P')\leq\frac{|V(Q)|-1}{2}(k-1).$$ For any path $P$ of $\Pa$ and any good path $Q'$ on $P$, no int-vertex of $Q'$ is 
 in $X^{k-1}_{\Pa}(P)$, therefore $|V(Q')\cap(|X^1_{\Pa}(P)\cup X^2_{\Pa}(P)\cup ...\cup X^{k-2}_{\Pa}(P))|\geq|V(Q')|-2\geq \frac{2}{k-1}f(G,\Pa)-1$. 
 Let $\Qa$ be a maximum set of non-intersecting good paths on $P$. By the definition, $t'_{\Pa}(P) = |\Qa|$, and we have 
 $$|X^1_{\Pa}(P)\cup X^2_{\Pa}(P)\cup ...\cup X^{k-2}_{\Pa}(P)|\geq |\cup_{Q\in \Qa}(V(Q)\cap(X^1_{\Pa}(P)\cup X^2_{\Pa}(P)\cup ...\cup X^{k-2}_{\Pa}(P)))|
 \geq$$ $$\geq \sum_{Q\in\Qa}(|V(Q)|-2)\geq t'_{\Pa}(P)\left( \frac{2}{k-1}f(G,\Pa)-1\right).$$
\end{proof}

\begin{corollary}
 \label{Cor1}
 Let $G$ be a connected graph and $\Pa \subseteq \La(G)$ with $|\Pa|=4$. Let $\Pa=\{P,P_1,P_2,P_3\}$ and let $Q$ be a good path on $\Pa$. 
 Then the following two statements hold:
 \begin{itemize}
  \item[(i)] $f(G,\Pa)\leq|V(Q)|-1$;
  \item[(ii)] $|X^1_{\Pa}(P)\cup X^2_{\Pa}(P)|\geq t'_{\Pa}(P)(f(G,\Pa)-1)$. 
 \end{itemize} 
\end{corollary} 

\begin{proof}
 The proof is the same as the proof of Lemma~\ref{X1X2} with respect to the following. Let $u,v$ be end-vertices of $Q$. Assume that $Q$ is 
 a $V(P_1)-V(P_2)$ path on $P$ (otherwise we renumber the paths) and we consider a vertex $x\in V(Q)\cap V(P_{3})$. 
 Then $$f(G,\Pa)\leq \sum_{P\in \Pa}d_{G}(x,P)=d_{G}(x,P_{1})+d_{G}(x,P_{2})\leq d_{G}(u,v)\leq |V(Q)|-1.$$ 
 Then we use Corollary~\ref{Cor1}(i) instead of Lemma~\ref{X1X2}(i) and the result follows.
\end{proof}
 
\noindent
\emph{Proof of Theorem~\ref{kpath}.}
 Suppose that $f(G,\Pa)\geq 1$. Hence $t'_{\Pa}(P)\geq 2$ by Lemma~\ref{Seg1}. Let $P \in \Pa$ be a path minimizing 
 $|X^1_{\Pa}(P)\cup X^2_{\Pa}(P)\cup ...\cup X^{k-2}_{\Pa}(P)|$. Let $\Pa-\{P\}=\{P_{1}, P_{2},...,P_{k-1}\}$ and $u_{i}, v_{i}$ be the end-vertices 
 of $P_{i}$ for $i\in\{1,2,...,k-1\}$. Assume that $Q$ is a good $V(P_1)-V(P_2)$ path on $P$ with end-vertices $u, v$ (otherwise we renumber paths 
 $P_{1}, P_{2},...,P_{k-1}$). Let $R$ be the shortest $\{u\}-V(P_{2})$ path on $P_{1}$ and $x \in V(R)\cap V(P_{2})$. We may assume that 
 $|V(u_{2}P_{2}v)|\leq|V(u_{2}P_{2}x)|$ (see Figure~\ref{fig1}). 
 
\begin{figure}[ht]
 \begin{center}
 \includegraphics[width=6cm]{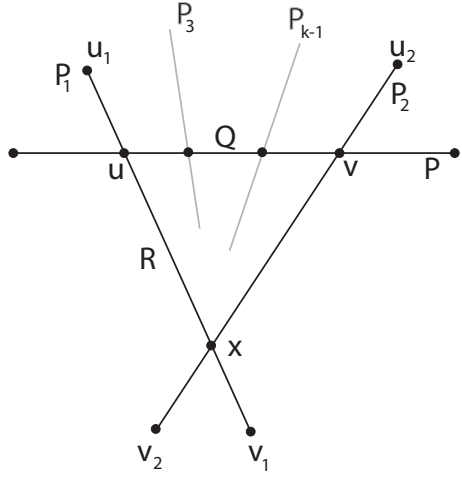}
 \end{center}\caption{A good $V(P_1)-V(P_2)$ path $Q$ and path $R$}\label{fig1}
\end{figure}
 
 We have $|V(R)|\geq 2$ from $f(G,\Pa)\geq 1$ and $|V(Q)|\geq \frac{2f(G,\Pa)}{k-1}+1$ from Lemma~\ref{X1X2}(i). Since $vQ$\emph{\v{u}} contains no vertex 
 of $V(P_{1})$, $vQuRx$ is a path in $G$. Futhermore, since \emph{\v{v}}$QuP_{1}$\emph{\v{x}} contains no vertex of $V(P_{2})$, 
 $S_{1}=v_{2}P_{2}vQuR$\emph{\v{x}}, $S_{2}=u_{2}P_{2}vQuRxP_{2}v_{2}$, and $S_{3}=u_{2}P_{2}xRuQ$\emph{\v{v}} are paths in $G$ (see Figure~\ref{fig2}).

\begin{figure}[ht]
 \begin{center}
 \includegraphics[width=12cm]{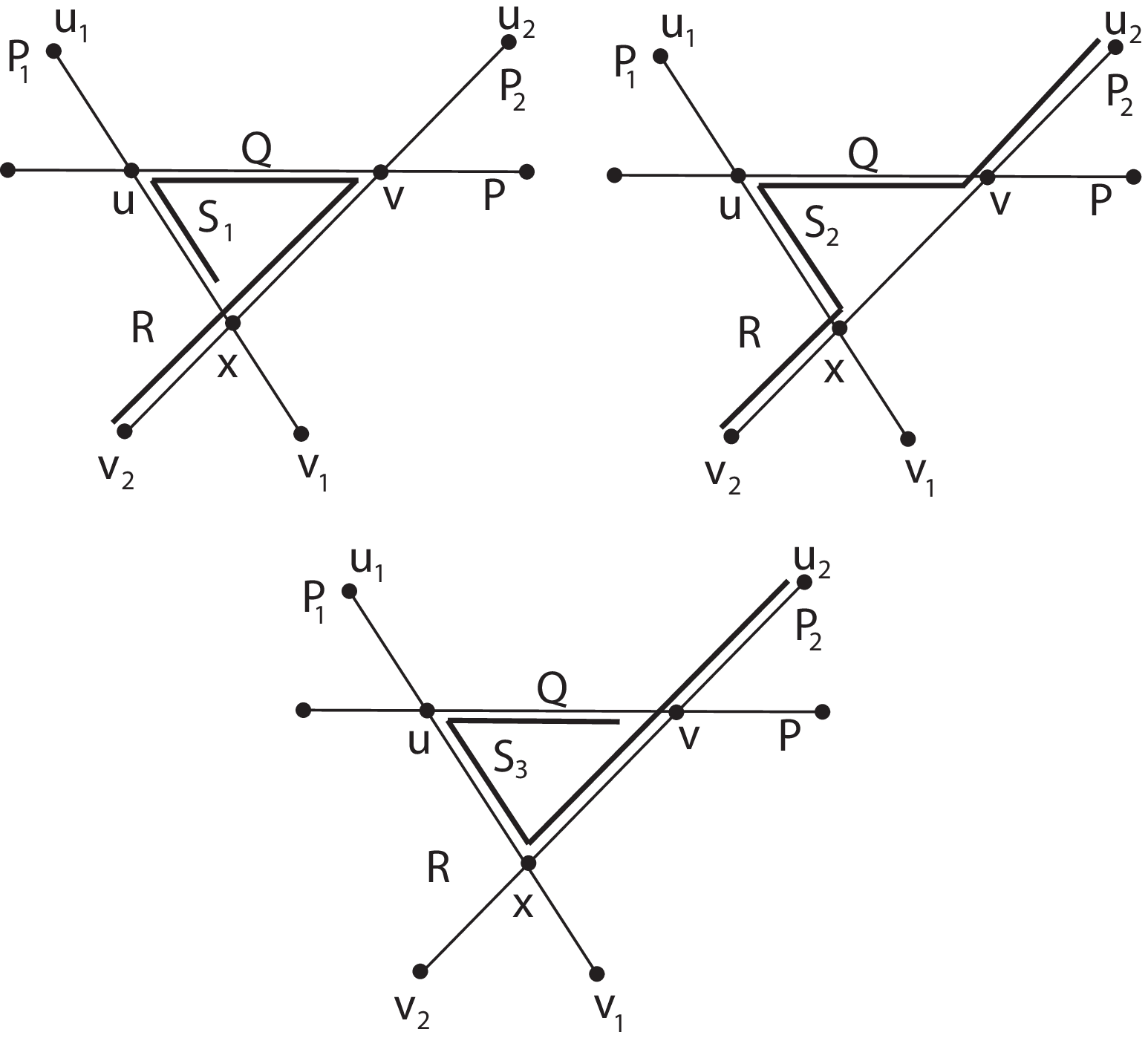}
 \end{center}
 \caption{Paths $S_{1}$, $S_{2}$, and $S_{3}$}
 \label{fig2}
\end{figure}

 By comparing the lengths of $P_{2}$ and $S_{1}$ and using Lemma~\ref{X1X2}(i) and $|V(R)|\geq 2$, we have 
 $$|V(u_{2}P_{2}v)|-1\geq |V(Q)|-1+|V(R)|-2\geq|V(Q)|-1\geq \frac{2f(G,\Pa)}{k-1}.$$
  
 Similarly for $P_{2}$ and $S_{2}$, we have 
 $$|V(vP_{2}x)|-1\geq |V(Q)|-1+|V(R)|-1\geq|V(Q)|\geq \frac{2f(G,\Pa)}{k-1}+1.$$
  
 Also for $P_{2}$ and $S_{3}$, we have 
 $$|V(xP_{2}v_{2})|-1\geq |V(Q)|-1+|V(R)|-2\geq|V(Q)|-1\geq \frac{2f(G,\Pa)}{k-1}.$$
  
 Therefore all together we have
 $$\ell(G)=|V(P_{2})|-1=|V(u_{2}P_{2}v)|-1+|V(vP_{2}x)|-1+|V(xP_{2}v_{2})|-1\geq$$
 $$\geq \frac{2f(G,\Pa)}{k-1}+\frac{2f(G,\Pa)}{k-1}+1+\frac{2f(G,\Pa)}{k-1}=\frac{6f(G,\Pa)}{k-1}+1.~~~~~~~~~~(*)$$

 Clearly $n_{i}=\frac{1}{i} \sum_{P'\in\Pa}X^i_{\Pa}(P')$. By the choice of $P$ and $t'_{\Pa}(P')\geq 2$ for every $P'\in \Pa$
 together with~$(*)$, Lemma~\ref{nbound}, and Lemma~\ref{X1X2} we have
 $$n\geq\frac{k\cdot \ell(G)+k+(k-2)\sum_{P'\in\Pa}X^1_{\Pa}(P')+\frac{k-3}{2}\sum_{P'\in\Pa}X^2_{\Pa}(P')+...+\frac{1}{k-2}\sum_{P'\in\Pa}X^{k-2}_{\Pa}(P')}
 {k-1}\geq$$  
 $$\geq \frac{k\cdot \ell(G)+k+\frac{1}{k-2}(\sum_{P'\in\Pa}X^1_{\Pa}(P')+\sum_{P'\in\Pa}X^2_{\Pa}(P')+...+\sum_{P'\in\Pa}X^{k-2}_{\Pa}(P'))}{k-1}\geq$$    
 $$\geq \frac{k\cdot \ell(G)+k+\frac{k}{k-2}(X^1_{\Pa}(P)+X^2_{\Pa}(P)+...+X^{k-2}_{\Pa}(P))}{k-1}\geq$$
 $$\geq\frac{k(\frac{6f(G,\Pa)}{k-1}+1)+k+\frac{2k}{k-2}(\frac{2}{k-1}f(G,\Pa)-1)}{k-1}=\frac{(6k^{2}-8k)f(G,\Pa)+2k^{3}-8k^{2}+6k}{(k-2)(k-1)^{2}},$$
 and hence $f(G,\Pa)\leq \frac{(k^{3}-4k^{2}+5k-2)n-2k^{3}+8k^{2}-6k}{6k^{2}-8k}$. This completes the proof of Theorem~\ref{kpath}.~~~~~$\square$
 
\bigskip
 
\noindent
\emph{Proof of Theorem~\ref{4path}.}
 We proceed as in the proof of Theorem~\ref{kpath} and use Corollary~\ref{Cor1}(i) instead of Lemma~\ref{X1X2}(i).
  
 By comparing the lengths of $P_{2}$ and $S_{1}$ and using Corollary~\ref{Cor1}(i) and $|V(R)|\geq 2$, we have 
 $$|V(u_{2}P_{2}v)|-1\geq |V(Q)|-1+|V(R)|-2\geq |V(Q)|-1\geq f(G,\Pa).$$
 
 Similarly for $S_{2}$ and $S_{3}$, we have 
 $$|V(vP_{2}x)|-1\geq |V(Q)|-1+|V(R)|-1\geq |V(Q)|\geq f(G,\Pa)+1,$$
 $$|V(xP_{2}v_{2})|-1\geq |V(Q)|-1+|V(R)|-2\geq |V(Q)|-1\geq f(G,\Pa).$$
  
 Therefore all together we have 
 $$\ell(G)=|V(P_{2})|-1=|V(u_{2}P_{2}v)|-1+|V(vP_{2}x)|-1+|V(xP_{2}v_{2})|-1\geq$$
 $$\geq f(G,\Pa)+f(G,\Pa)+1+f(G,\Pa)=3f(G,\Pa)+1.~~~~~~~~~~(**)$$
 
 By the choice of $P$ and $t'_{\Pa}(P')\geq 2$ for every $P'\in \Pa$ together with~$(**)$, Lemma~\ref{X1X2}, and Lemma~\ref{Seg1} we have
 $$n\geq\frac{4\ell(G)+4+2\sum_{P'\in\Pa}X^1_{\Pa}(P')+\frac{1}{2}\sum_{P'\in\Pa}X^2_{\Pa}(P')}{3}\geq$$ 
 $$\geq\frac{4(3f(G,\Pa)+1)+4+4(f(G,\Pa)-1)}{3}=\frac{16f(G,\Pa)+4}{3},$$  
 and hence $f(G,\Pa)\leq \frac{3n-4}{16}$. This completes the proof of Theorem~\ref{4path}.~~~~~~~~~~~~~~~~~~~~~~~~~~~$\square$  

\section{Conclusion}
As it was mentioned in Introduction, we extend Conjecture~\ref{3path} to Conjecture~\ref{3-6path}.

\begin{conjecture}
 \label{3-6path}
 For every connected graph, any $k$ of its longest paths have a common vertex for $3\leq k\leq 6$.
\end{conjecture}

Conjecture~\ref{sublin} is an extension of a Conjecture stated in~\cite{FuFuNaOz} for $3$ longest paths. We prove that Conjecture~\ref{sublin} 
is equivalent with Conjecture~\ref{3-6path}.

\begin{conjecture}
 \label{sublin}
 There exists a sublinear function $g$ such that for every connected graph $G$ of order $n$ and every subset $\Pa$ of $\La(G)$ with $3\leq |\Pa|\leq 6$, 
 $f(G,\Pa)\leq g(n)$.
\end{conjecture}

Let $\mathcal{G}_n$ be a class of connected graphs of order at least $n$. In other words, using $d_k(\mathcal{G}_n)$ with $3\leq k \leq 6$, 
Conjecture~\ref{sublin} translates into the following statement. The path distance ratio $d_k(\mathcal{G}_n)$ goes to 0 as $n$ goes to infinity.

\begin{theorem}
 \label{equivalent}
 Conjecture~\ref{3-6path} is true if and only if Conjecture~\ref{sublin} is true.
\end{theorem}

\begin{proof}
 Suppose Conjecture~\ref{3-6path} holds. For every set $\Pa$ of $k$ longest paths ($3\leq k\leq 6$) of every connected graph $G$, we have $f(G,\Pa) = 0$.
 Thus any non-negative sublinear function implies that Conjecture~\ref{sublin} holds.

 Suppose Conjecture~\ref{sublin} holds. We prove the contrapositive statement, that is, if Conjecture~\ref{3-6path} is not true, then neither is 
 Conjecture~\ref{sublin}. For $3\leq k\leq 6$, we consider a connected graph $G$ and a set $\Pa$ of its $k$ longest paths so that they have no common vertex. 
 We extend $G$ by adding a pendant edge to every vertex, which is an end-vertex of a path of $\Pa$, and we note that each path of $\Pa$ prolonged with two 
 of these new edges is a longest path in the extended graph. For a non-negative integer $t$, we subdivide every edge of the extended graph $t$ times and 
 we observe that the corresponding $k$ paths, say $\Pa_t$, are longest paths in the resulting graph $G_t$. Let $n$ be the number of vertices and $m$ the
 number of edges of $G$. We see that $G_t$ has at most $n + t(m +2k)$ vertices. By construction, $f(G_t,\Pa_t)\geq t$. We consider the sequence of graphs 
 $(G_t)_{t=1}^{\infty}$ and we note that $f(G_t,\Pa_t)$ cannot be bounded from above by a sublinear function.
\end{proof}

\section*{Acknowledgements}
This work was partly supported by the project LO1506 of the Czech Ministry of Education, Youth and Sports.

The first and third authors were supported by project GA14-19503S of the Grant Agency of the Czech Republic. 

The second author's research is supported by Grant-in-Aid for Scientic Research (C) (15K04979). Also, this work was partially completed during a visit 
of the second author to the University of West Bohemia. He wishes to express his thanks for the generous hospitality.


\begin{thebibliography}{99}
 \bibitem{axen} M. Axenovich, When do $3$ longest paths have a common vertex? Discrete Math. Alg. Appl. 1 (2009) 115-120.
  
 \bibitem{bali} P. Balister, E. Gy\"ori, J. Lehel, R. Schelp, Longest paths in circular arc graphs, Combin. Probab. Comput. 13 (2004) 311-317.
 
 \bibitem{reze} S. F. de Rezende, C. G. Fernandes, D. M. Martin, Y. Wakabayashi, Intersecting longest paths, Discrete Math. 313 (2013) 1401-1408.

 \bibitem{gala} P. Erd\"os, G. Katona (Eds.), Theory of Graphs, Proceedings of the Colloquium Hald at Tihany, Hungary, 1966, Academic Press, New York, 1968,   
                Problem 4 (T. Gallai), p.362.
                
 \bibitem{FuFuNaOz} S. Fujita, M. Furuya, R. Naserasr, K. Ozeki, A New Approach Towards a Conjecture on Intersecting Three Longest Paths, 
                    arXiv:1503.01219v2, 2015.
 
 \bibitem{klav} S. Klav\v zar, M. Petkov\v sek, Graphs with nonempty intersection of longest paths, Ars Combin. 26 (1990) 43-52.
 
 \bibitem{skup} Z. Skupie\'n, Smallest sets of longest paths with empty intersection. Combin. Probab. Comput. 5 (1996), no. 4, 429-436. 
 
 \bibitem{thom} C. Thomassen, Planar and infinite hypohamiltonian and hypotraceable graphs, Discrete Math. 14 (1976) 377-389.
 
 \bibitem{walt} H. Walther, \"Uber die Nichtexistenz eines Knotenpunktes durch den allen l\"angsten Wege eines Graphen gehen, J. Combin. Theory 6 (1969) 1-6.

 \bibitem{waltvoss} H. Walther, H. J. Voss, \"Uber Kreise in Graphen, VED Deutcher Verlag der Wissenschaften, 1974.

 \bibitem{zamf} T. Zamfirescu, On longest path and circuits in graphs, Math. Scand. 38 (1976) 211-239.

 \bibitem{zamf2} T. Zamfirescu, Intersecting longest path or cycles: short survey, An. Univ. Craiova Ser. Mat. Inform. 28 (2001) 1-9. 
\end{thebibliography}
\end{document}